\documentclass[10pt,a4paper]{article}
\usepackage[latin1]{inputenc}
\usepackage{amsmath}
\usepackage{amsfonts}
\usepackage{amssymb}
\usepackage{graphicx}
\usepackage{amsthm}
\theoremstyle{definition}
\newtheorem{defn}{Definition}[section]
\newtheorem{thm}{Theorem}[section]

\newtheorem{exmp}{Example}[section]
\newtheorem{pro}{Proposition}[section]
\theoremstyle{remark}
\newtheorem{rem}{Remark}[section]
\theoremstyle{corollary}

\numberwithin{equation}{section}
\date{}
 \title{\textbf{Caristi-Banach type contraction via simulation function}}
 \author{Swati Antal and U. C. Gairola}
\begin{document}
\maketitle
\begin{center}
	Department of Mathematics\\ H. N. B. Garhwal University, BGR Campus Pauri Garhwal-246001 Uttarakhand, India.\\
		E-mail: antalswati11@gmail.com, ucgairola@rediffmail.com
\end{center}	
\textbf{Abstarct :}  In this paper we introduce the notion of a Caristi-Banach type $ \mathcal{Z}_{\mathcal{R}}^{b} $- contraction in the framework of $ b $-metric space endowed with a transitive relation that combine the ideas of  Caristi type contraction and Banach contraction with a help of simulation function. We present an example to clarify the statement of the given result.

\section{Introduction and preliminaries}
In $ 1922 $, Polish mathematician Stefan Banach \cite{sb} gave a fixed point theorem. It is also known as the Banach Contraction mapping theorem or principle (BCP).  It is an important tool in the metric fixed point theory. It confirms the existence and uniqueness of fixed point of certain self maps of metric spaces and provides a constructive method to find fixed points. There are so many extension, generalizations of BCP in different settings and there applications. Among them, In 1976, Caristi \cite{c} proved a fixed point theorem and applied to derive a generalization of the Contraction Mapping Principle in a complete metric space. Recently, In 2019, E. Karapinar et al., \cite{kp} give a new fixed point theorem in $ b $ - metric space which is inspired from both Caristi and Banach. $ b $ - metric space introduced by Czerwik \cite{cz} to generalize the concept of metric space by introducing a real number $s \geq 1$ in the triangle inequality of metric space.

Inspired by E. Karapinar et al., \cite{kp} we introduce the notion of a Caristi-Banach type $ \mathcal{Z}_{\mathcal{R}}^{b} $- contraction in the framework of $ b $-metric space endowed with a transitive relation that combine the ideas of  Caristi type contraction and Banach contraction with a help of simulation function. We present an example also to clarify the statement of the given result.

\begin{defn}
	\cite{cz} Let $ M $ be a non-empty set and $s \geq 1$ be a given real number. A function $d : M \times M \rightarrow [0, \infty)$ is said to be a $ b $-metric space if, for all $\sigma, \rho, w \in M$, the following conditions are satisfied:
	\begin{flushleft}
		\hspace{.2cm}	(i) $ d(\sigma, \rho) = 0 $ \text{iff} $ \sigma = \rho $;\hspace{9cm} \\\vspace{.1cm} 
		\hspace{.2cm}	(ii) $ d(\sigma, \rho) = d(\rho, \sigma) $;\\\vspace{.1cm}
		\hspace{.2cm}	(iii) $ d(\sigma, w)  \leq  s[d(\sigma, \rho) + d(\rho, w)] $.
	\end{flushleft}
	The pair $ (M, d) $ is called a $ b $-metric space.
\end{defn}
It should be noted that, every metric space is a $ b $-metric space with $ s = 1 $ and hence the class of $ b $-metric spaces is larger than the class of metric spaces. But a metric space need not be $ b $-metric space (see example 1.4 \cite{jr2}).
\begin{defn}\cite{mb}
	Let $ (M, d) $ be a $ b $-metric space. 
	\begin{flushleft}
		\hspace{.2cm}	(i) A sequence $\{\sigma_{n}\}$ in $ M $ is called $ b $-convergent if and only if there exist $\sigma \in M$ such that $d(\sigma_{n}, u) \rightarrow 0$, as $n \rightarrow \infty$. In this case, we write $\lim_{n\to\infty}\sigma_{n} \rightarrow u$.\\
		\hspace{.2cm}	(ii) $\{\sigma_{n}\}$ in $ M $ is said to be $ b $-Cauchy if and only if $d(\sigma_{n}, \sigma_{m}) \rightarrow 0$, as $n, m \rightarrow \infty$.\\\vspace{.1cm}
		\hspace{.2cm} (iii) The $ b $-metric space $ (M, d) $ is said to be $ b $-complete if every $ b $-Cauchy sequence $\{\sigma_{n}\}$ in $ M $ is convergent.	
	\end{flushleft}
\end{defn}
Recently, in 2015, Khojasteh et al. \cite{1} introduced the notion of simulation function with a view to consider a new class of contractions, called $\mathcal{Z}$-contraction with respect to a simulation function.
\begin{defn}\cite{1}
	A mapping $\zeta : [0, \infty) \times [0, \infty) \rightarrow \mathbb{R}$ is a simulation function if: 
	\begin{flushleft}
		\hspace{.2cm}	$ (\zeta_{1}) $ \hspace{.1cm} $ \zeta(0, 0) = 0 $;\\
		\hspace{.2cm}	$ (\zeta_{2}) $ \hspace{.1cm} $\zeta(\mathfrak{t}, \mathfrak{s}) < \mathfrak{s} - \mathfrak{t} $, \hspace{.1cm} $ \mathfrak{s}, \mathfrak{t} > 0 $;\\
		\hspace{.2cm}	$ (\zeta_{3}) $ \hspace{.1cm}  $ \{\mathfrak{t}_{n}\}\hspace{.1cm}$ and $\{\mathfrak{s}_{n}\} $ are sequences in $ (0, \infty) $ satisfying $ \lim_{n\to\infty} \mathfrak{t}_{n} = \lim_{n\to\infty}\mathfrak{s}_{n} > 0 $, then $ \limsup_{n\to\infty}\zeta(\mathfrak{t}_{n}, \mathfrak{s}_{n}) < 0 $.
	\end{flushleft}	
	Set of all simulation functions is  denoted by $\mathcal{Z}$.  For examples of simulation function we may refer to (\cite{10}, \cite{11}, \cite{1}).
\end{defn}

In what follows $ (\mathcal{M}, d) $, $\mathcal{R}$, $\mathbb{N}$ and $\mathbb{N}_{0}$ respectively, stand for a metric space, a non-empty binary relation defined on a non-empty set $ \mathcal{M} $, the set of natural numbers and the set of whole numbers.
\begin{defn}\cite{3}
	A binary relation $ \mathcal{R} $ on a non-empty set $ \mathcal{M} $ is defined as a subset of $\mathcal{M} \times \mathcal{M}$. We say that \textquotedblleft $ \sigma $ is $\mathcal{R}$-related to $ \rho \textquotedblright$ iff $(\sigma, \rho) \in \mathcal{R}$.
\end{defn}

\begin{defn}\cite{4}
	$\mathcal{R}$ is complete if either $(\sigma, \rho) \in \mathcal{R}$ or $(\rho, \sigma) \in \mathcal{R}$ (i.e. $[\rho, \sigma] \in \mathcal{R}$), $ \forall $ $ \sigma, \rho \in \mathcal{M} $.
\end{defn}
\begin{defn}\cite{5}
	Let $ \mathcal{F} $ be a self-mapping defined on a non-empty set $ \mathcal{M} $. Then $\mathcal{R}$ is $ \mathcal{F} $-closed if
	\begin{align*}
	(\sigma, \rho) \in \mathcal{R} \Rightarrow (\mathcal{F}\sigma, \mathcal{F}\rho) \in \mathcal{R},\hspace{.2cm} \sigma, \rho \in \mathcal{M}.
	\end{align*}
\end{defn}
\begin{pro}\cite{5}
	If $\mathcal{R}$ is $ \mathcal{F} $-closed, then $\mathcal{R}^{s}$ is also $ \mathcal{F} $-closed. 
\end{pro}
\begin{defn}\cite{5}
	A sequence $\{\sigma_{n}\}$ in $ \mathcal{M} $ is $\mathcal{R}$-preserving if
	\begin{align*}
	(\sigma_{n}, \sigma_{n+1}) \in \mathcal{R}, \hspace{.2cm} n \in \mathbb{N}_{0}.
	\end{align*}
\end{defn}
\begin{defn}
	$\mathcal{R}$ is transitive if $(\sigma, \rho) \in \mathcal{R}$ and $(\rho, \eta) \in \mathcal{R}$ implies that $(\sigma,\eta) \in \mathcal{R}$
\end{defn}
\begin{defn} \cite{sw}
	$\mathcal{R}$  is  $ b-d $-self-closed if $\{\sigma_{n}\}$ is an $\mathcal{R}$-preserving sequence and if
	\begin{align*}
	\sigma_{n} \xrightarrow{d} \sigma \hspace{.2cm} as \hspace{.2cm} n \rightarrow \infty,
	\end{align*}
	then there exists a subsequences $\{\sigma_{n_{k}}\}$ of $\{\sigma_{n}\}$ with $[\sigma_{n_{k}}, \sigma] \in \mathcal{R}$, $k \in \mathbb{N}$.
\end{defn}
\begin{defn}\cite{7}
	A subset $ D $ of $ \mathcal{M} $ is $\mathcal{R}$-directed if for each pair of points $\sigma, \rho \in D$, there exists $\eta \in \mathcal{M}$ satisfying $(\sigma,\eta) \in \mathcal{R}$ and $(\rho,\eta) \in \mathcal{R}$.
\end{defn}
\begin{defn}\cite{9}
	For $\sigma, \rho \in \mathcal{M}$, a path of length $ k $ in $\mathcal{R}$ from $ \sigma $ to $ \rho $ is a finite sequence $\{\eta_{0},{\eta}_{1},{\eta}_{2},...,{\eta}_{k}\} \subset \mathcal{M}$ satisfying:\\
	(i) $ {\eta}_{0} = \sigma$ and $ {\eta}_{k} = \rho$,\\
	(ii) $({\eta}_{i}, {\eta}_{i + 1}) \in \mathcal{R}$ for each $ i $ $(0 \leq i \leq k - 1)$ ($ k $ is a natural number).\\
	Clearly a path of length $ k $ necessitate $ k + 1 $ elements of $ \mathcal{M} $, which are not essentially distinct.
\end{defn}
In the following \vspace{.2cm}\\
$ \mathcal{M}(\mathcal{F}; \mathcal{R}) := \{\sigma \in \mathcal{M} : (\sigma, \mathcal{F}\sigma) \in \mathcal{R}\} $, where $ \mathcal{F} : \mathcal{M} \rightarrow \mathcal{M} $ and $ \gamma(\sigma, \rho,\mathcal{R}) $ is the class of all paths in $ \mathcal{R} $ from $ \sigma $ to $ \rho $.

\section{Main Result}
\begin{defn}
	Let $ \mathcal{F} $ be a self mapping on a $ b $-metric space $ (\mathcal{M}, d) $ equipped with a binary relation $ \mathcal{R} $. If there exist $ \zeta \in \mathcal{Z}_{\mathcal{R}}^{b} $ and $ \mathcal{\phi} : \mathcal{M} \rightarrow [0, \infty)$ such that
	\begin{align}\label{1}
	d(\sigma, \mathcal{F}\sigma) > 0 \Rightarrow \zeta(s d(\mathcal{F}\sigma, \mathcal{F}\rho), (\mathcal{\phi}(\sigma) - \mathcal{\phi}(\mathcal{F}\sigma) d(\sigma, \rho)) \geq 0,
	\end{align}
	$ \forall $ $ \sigma, \rho \in \mathcal{M} $, $ (\sigma, \rho) \in \mathcal{R} $, then $ \mathcal{F} $ is called  Caristi - Banach type $ \mathcal{Z}_{\mathcal{R}}^{b} $-contraction.
\end{defn}	
\begin{thm}\label{a}
	Let $ (\mathcal{M}, d) $ be a complete $ b $- metric space equipped with a binary relation $ \mathcal{R} $ and $ \mathcal{F} $ be a self mapping on $ \mathcal{M} $. Let the following hypotheses holds\\
	(i) $ \mathcal{M}(\mathcal{F}; \mathcal{R}) $ is non-empty;\\
	(ii) $ \mathcal{R} $ is $ \mathcal{F} $-closed and transitive;\\
	(iii) either $ \mathcal{F} $ is $ \mathcal{R} $- continuous or $ \mathcal{R} $ is $b$ - $d$- self-closed;\\
	(iv) $ \mathcal{F} $ is  Caristi - Banach type $ \mathcal{Z}_{\mathcal{R}}^{b} $-contraction with respect to $ \zeta \in \mathcal{Z} $.\\
	Then $ \mathcal{F} $ has a fixed point.
\end{thm}
\begin{proof}
	Let $ \sigma_{0} $ be an arbitrary point in $ \mathcal{M}(\mathcal{F}; \mathcal{R}) $. Put $ \sigma_{n} = \mathcal{F}\sigma_{n - 1} = \mathcal{F}^{n} \sigma_{0} $ $ \forall n \in \mathbb{N} $. Let $ C_{n + 1} = d(\sigma_{n}, \sigma_{n + 1}) $,  if for some $ n' \in \mathbb{N}_{0} $, $ \sigma_{n'} = \sigma_{n' + 1} $, then $ \sigma_{n'} $ is a fixed point of $ \mathcal{F} $ and so the proof is complete. Thus, we let $ \sigma_{n} \neq \sigma_{n + 1} $ $ \forall $ $ n \in \mathbb{N}_{0} $ i.e., $ C_{n + 1} > 0 $. Since $ (\sigma_{0}, \mathcal{F}\sigma_{0}) \in \mathcal{R} $,  using the $ \mathcal{F} $-closedness of $ \mathcal{R} $, we obtain
		\begin{align*}
		(\mathcal{F}\sigma_{0}, \mathcal{F}^{2}\sigma_{0}), (\mathcal{F}^{2}\sigma_{0}, \mathcal{F}^{3}\sigma_{0}), ... , (\mathcal{F}^{n}\sigma_{0}, \mathcal{F}^{n + 1}\sigma_{0}), ... \in \mathcal{R}.
		\end{align*} 
		Thus
		\begin{align}\label{2}
		(\sigma_{n}, \sigma_{n + 1}) \in \mathcal{R}, 
		\end{align}
	and the sequence $\{\sigma_{n}\}$ is $\mathcal{R}$- preserving. Since $ \mathcal{F} $ is Caristi - Banach type $ \mathcal{Z}_{\mathcal{R}}^{b} $-contraction, we have
		\begin{eqnarray*}
		0 &\leq&	\zeta \big( s d(\mathcal{F}\sigma_{n - 1}, \mathcal{F}\sigma_{n}), (\mathcal{\phi}(\sigma_{n - 1}) - \mathcal{\phi}(\mathcal{F}\sigma_{n - 1}))(\sigma_{n - 1}, \sigma_{n}) \big)\\ &<& (\mathcal{\phi}(\sigma_{n - 1}) - \mathcal{\phi}(\mathcal{F}\sigma_{n - 1}))(\sigma_{n - 1}, \sigma_{n}) -  s d(\mathcal{F}\sigma_{n - 1}, \mathcal{F}\sigma_{n}),
		\end{eqnarray*}
	which yield
	\begin{eqnarray*}
	C_{n + 1} = d(\sigma_{n}, \sigma_{n + 1}) &=& d(\mathcal{F}\sigma_{n - 1}, \mathcal{F}\sigma_{n}) \leq s d(\mathcal{F}\sigma_{n - 1}, \mathcal{F}\sigma_{n})\\ &<&  (\mathcal{\phi}(\sigma_{n - 1}) - \mathcal{\phi}(\mathcal{F}\sigma_{n - 1}))(\sigma_{n - 1}, \sigma_{n})\\ &=&
	(\mathcal{\phi}(\sigma_{n - 1}) - \mathcal{\phi}(\sigma_{n}))C_{n}.
	\end{eqnarray*}
So we have
\begin{center}
$ 0 < \frac{C_{n + 1}}{C_{n}} \leq (\mathcal{\phi}(\sigma_{n - 1}) - \mathcal{\phi}(\sigma_{n})) $ for each $ n \in \mathbb{N} $.
\end{center}	
	Thus the sequence $ \{\mathcal{\phi}(\sigma_{n})\} $ is necessarily non-negative and decreasing. Hence, it converges to some $ \mathfrak{a} \geq 0$. On the other hand, for each $ n \in \mathbb{N} $, we have
	\begin{eqnarray*}
	\sum_{k = 1}^{n} \frac{C_{k + 1}}{C_{k}} &\leq& \sum_{k = 1}^{n} (\mathcal{\phi}(\sigma_{k - 1}) - \mathcal{\phi}(\sigma_{k}))\\ &=& (\mathcal{\phi}(\sigma_{0}) - \mathcal{\phi}(\sigma_{1})) + (\mathcal{\phi}(\sigma_{1}) - \mathcal{\phi}(\sigma_{2})) +...+ (\mathcal{\phi}(\sigma_{n - 1}) - \mathcal{\phi}(\sigma_{n}))\\
	&=& (\mathcal{\phi}(\sigma_{0}) - \mathcal{\phi}(\sigma_{n})) \rightarrow \mathcal{\phi}(\sigma_{0}) - \mathfrak{a} < \infty, ~ as ~ n \rightarrow \infty.
	\end{eqnarray*}
It means that
\begin{align*}
	\sum_{n = 1}^{\infty} \frac{C_{n + 1}}{C_{n}} < \infty.
\end{align*}	
Accordingly, we have
\begin{align}\label{3}
	\lim_{n\to\infty} \frac{C_{n + 1}}{C_{n}} = 0.
\end{align}	
On account of \eqref{3}, for $ \varrho \in (0, 1) $, there exist $ n_{0} \in \mathbb{N} $ such that
\begin{align}\label{4}
	\frac{C_{n + 1}}{C_{n}} \leq \varrho,~ \forall ~ n \geq n_{0}.
\end{align}	
 It yield that	
\begin{align}\label{5}
d(\sigma_{n}, \sigma_{n + 1}) \leq \varrho d(\sigma_{n - 1}, \sigma_{n}),~ \forall~ n \geq n_{0}.
\end{align}
Now using Lemma 3.1 \cite{sl} we obtain that the sequence $ \{\sigma_{n}\} $ is Cauchy. Thus from the completeness of $ \mathcal{M} $, there exist $ \sigma \in \mathcal{M} $ such that $ \sigma_{n} \rightarrow \sigma $ as $ n \rightarrow \infty $. By (iii), if $ \mathcal{F} $ is $ \mathcal{R} $- continuous then $ \mathcal{F}\sigma_{n} \rightarrow \mathcal{F} \sigma $ as $ n \rightarrow \infty $.\\

Alternately, let us assume that $ \mathcal{R} $ is $b$ - $d$-self-closed. As $ \{\sigma_{n}\} $ is an $ \mathcal{R} $ preserving sequence and $ \sigma_{n} \xrightarrow{d} \sigma $ as $ n \rightarrow \infty $. So there exist a subsequence $ \{\sigma_{n_{k}}\} $ of $ \{\sigma_{n}\} $ with $ [\sigma_{n_{k}}, \sigma] \in \mathcal{R} $, $ \forall $  $ k \in \mathbb{N}_{0} $. Notice that $ [\sigma_{n_{k}}, \sigma] \in \mathcal{R} $, $ \forall $  $ k \in \mathbb{N}_{0} $ implies that either $ (\sigma_{n_{k}}, \sigma) \in \mathcal{R} $, $ \forall $  $ k \in \mathbb{N}_{0} $ or $(\sigma, \sigma_{n_{k}}) \in \mathcal{R} $, $ \forall $  $ k \in \mathbb{N}_{0} $. Applying condition (iv) to $ (\sigma_{n_{k}}, \sigma) \in \mathcal{R} $, $ \forall $  $ k \in \mathbb{N}_{0} $, we have 
\begin{eqnarray*}
0 &\leq& \zeta \big(s d(\mathcal{F}\sigma_{n_{k}}, \mathcal{F}\sigma),  (\mathcal{\phi}(\sigma_{n_{k}}) - \mathcal{\phi}(\mathcal{F}(\sigma_{n_{k}})) d(\sigma_{n_{k}}, \sigma) \big)\\ &<& (\mathcal{\phi}(\sigma_{n_{k}}) - \mathcal{\phi}(\mathcal{F}(\sigma_{n_{k}})) d(\sigma_{n_{k}}, \sigma) - s d(\mathcal{F}\sigma_{n_{k}}, \mathcal{F}\sigma)
\end{eqnarray*}
 $ \Longrightarrow $  ~~~~~~~~~~~ $  s d(\mathcal{F}\sigma_{n_{k}}, \mathcal{F}\sigma) < (\mathcal{\phi}(\sigma_{n_{k}}) - \mathcal{\phi}(\mathcal{F}(\sigma_{n_{k}})) d(\sigma_{n_{k}}, \sigma) $.\\
 
 Using the triangle inequality  together with the inequality above, we derive that
 \begin{eqnarray*}
 d(\sigma, \mathcal{F}\sigma) &\leq& s [d(\sigma, \sigma_{n_{k} + 1}) + d(\sigma_{n_{k} + 1}, \mathcal{F}\sigma)]\\
 &=& s [d(\sigma, \sigma_{n_{k} + 1}) + d(\mathcal{F}\sigma_{n_{k}}, \mathcal{F}\sigma)]\\
 &<& s d(\sigma, \sigma_{n_{k} + 1}) + (\mathcal{\phi}(\sigma_{n_{k}}) - \mathcal{\phi}(\mathcal{F}(\sigma_{n_{k}})) d(\sigma_{n_{k}}, \sigma),
 \end{eqnarray*}
 putting $ n \rightarrow \infty $ and using $ \sigma_{n_{k}} \xrightarrow{d} \sigma $, above inequality $ \rightarrow 0 $  as  $ n \rightarrow \infty $.
Consequently, we obtain that $ d(\sigma, \mathcal{F}\sigma) = 0 $, i.e., $ \mathcal{F}\sigma = \sigma $.\\
Similarly, if $(\sigma, \sigma_{n_{k}}) \in \mathcal{R} $, $ \forall $  $ k \in \mathbb{N}_{0} $, we obtain $ d(\mathcal{F}\sigma, \sigma) = 0 $. So that $ \mathcal{F}\sigma = \sigma $, i.e., $ \sigma $ is a fixed point of $ \mathcal{F} $.		
\end{proof}	
\begin{thm}\label{b}
	 In addition to the hypotheses of Theorem \ref{a}, if
	 \begin{flushleft}
	  (v) $ \gamma(\sigma, \rho, \mathcal{R}) \neq \phi  $.
	 \end{flushleft}
Then $ \mathcal{F} $ has a unique fixed point.
\end{thm}
\begin{proof}
	 Let $ \sigma^*, \rho^* $ are two fixed point of $ \mathcal{F} $ such that $ \sigma^* \neq  \rho^* $. Since $ \gamma(\sigma^*, \rho^*, \mathcal{R}) \neq \phi$, there exists a path $ ( \{\eta_0, \eta_1, \eta_2,...,\eta_k\}) $ of some finite length $ k $ in $ \mathcal{R} $ from $ \sigma $ to $ \rho $ so that 
	 \begin{align*}
	 \eta_0 = \sigma^*, \eta_k = \rho^*, (\eta_{i}, \eta_{i + 1}) \in \mathcal{R}, i = 0, 1, 2,..., k-1.
	 \end{align*}
	 Since $ \mathcal{R} $ is transitive,
	 \begin{align*}
	 (\eta_0, \eta_k) \in \mathcal{R}.
	 \end{align*}
	 Therefore
	 \begin{eqnarray*}
	 0 &\leq& \zeta \big(s d(\mathcal{F}\eta_0, \mathcal{F}\eta_k), (\mathcal{\phi}(\eta_0) - \mathcal{\phi}(\mathcal{F}\eta_0)) d(\eta_0, \eta_k) \big)\\
	 &<& (\mathcal{\phi}(\eta_0) - \mathcal{\phi}(\mathcal{F}\eta_0)) d(\eta_0, \eta_k) - s d(\mathcal{F}\eta_0, \mathcal{F}\eta_k)\\
	 &=& (\mathcal{\phi}(\sigma^*) - \mathcal{\phi}(\mathcal{F}\sigma^*)) d(\sigma^*, \rho^*) - s d(\mathcal{F}\sigma^*, \mathcal{F}\rho^*) < 0,
	 \end{eqnarray*}
	 which is a contradiction. Thus $ \mathcal{F} $ has a unique fixed point.
\end{proof}
\begin{exmp}\label{ex1}
	Let $ (\mathcal{M}, d) = [1, 4] $ and $ d(\sigma, \rho) = (\sigma - \rho)^{2}$. Then $ (\mathcal{M}, d, s) $ be a complete $ b $ - metric space with coefficient $ s = 2 $. Define a binary relation 
	\begin{align*}
		 \mathcal{R} = \{ (1, 1), (1, 2), (1, 3), (1, 4), (2, 1), (2, 2), (2, 3), (2, 4), (3, 1), (3,2), (3, 3), (3, 4) \} 
	\end{align*} on $ \mathcal{M} $ and the mapping $ \mathcal{F} :  \mathcal{M} \rightarrow \mathcal{M} $ by
\begin{center}
		$ \mathcal{F}(\sigma) $ = $ \begin{cases}
		1, ~ \text{if} ~1\leq \sigma \leq 2;\\
		2, ~\text{if}~2 < \sigma \leq 3;\\
		3, ~\text{if}~3 < \sigma \leq 4.
	\end{cases} $
\end{center}
Let $ \phi : \mathcal{M} \rightarrow [0, \infty) $ defined by $ \phi (\sigma) = 3 \sigma $, $ \sigma \in [0, \infty) $. Since for $ (\sigma, \rho) \in \mathcal{R} $ we have $ (\mathcal{F}\sigma, \mathcal{F}\rho) \in \mathcal{R} $ which implies that $ \mathcal{R} $ is $ \mathcal{F} $ - closed and transitive. For $ \sigma = 2, \mathcal{F}\sigma = 1 $, $ (\sigma, \mathcal{F}\sigma) \in \mathcal{R} $ i.e., $ \mathcal{M}(\mathcal{F}; \mathcal{R}) \neq \phi $. If we take any $ \mathcal{R} $- preserving sequence $ \{\sigma_{n}\} $ with $ \sigma_{n} \xrightarrow{d} \sigma $ and $ (\sigma_{n}, \sigma_{n + 1}) \in \mathcal{R} $, $ \forall $ $ n \in \mathbb{N}_{0} $, this implies that there exists an integer $ N \in \mathbb{N}_{0} $ such that  $ \sigma_n = \sigma $ for all $ n \geq N $. So, we can take a subsequence $ \{\sigma_{n_{k}}\} $ of the sequence $ \{\sigma_n\} $ such that $ \sigma_{n_{k}} = \sigma $ for all $ k \in \mathbb{N}_{0}, $. which amounts to saying that $ [\sigma_{n_k}, \sigma] \in \mathcal{R} $, for all $ k \in \mathbb{N}_0$. Therefore, $ \mathcal{R} $ is $ b-d $-self-closed.\\\vspace{.01cm}
 	\hspace{.2cm} Now, with a view to check that $ \mathcal{F} $ is Caristi - Banach type $ \mathcal{Z}_\mathcal{R}^{b} $-contraction, let for all  $ \sigma \in \mathcal{M} $ such that $ d(\sigma, \mathcal{F}\sigma) > 0 $ and $ (\sigma, \rho) \in \mathcal{R} $, (in this example $ \sigma \neq 1 $), we have
 	\begin{eqnarray*}
 		0 &\leq& \zeta \big(s d(\mathcal{F}\sigma, \mathcal{F}\rho), (\mathcal{\phi}(\sigma) - \mathcal{\phi}({\mathcal{F}\sigma})) d(\sigma, \rho) \big).
 	\end{eqnarray*}
Thus all the hypotheses of Theorem \ref{a} (\ref{b}) are verified. Hence $ \sigma = 1 $ is the unique fixed point of $ \mathcal{F} $.
\end{exmp}
\begin{rem}
	It is interesting to see that in Example \ref{ex1} at (2, 4), it is not a b-simulation function \cite{dm} i.e.,
	\begin{eqnarray*}
	0 &\leq& \zeta \big(2 d(\mathcal{F}2, \mathcal{F}4), d(2, 4) \big)\\
	&<& d(2, 4) - 2 d(1, 3) = -4 < 0,
	\end{eqnarray*} 
	which is a contradiction and if we take the usual metric on place of $ d(\sigma, \rho) = (\sigma - \rho)^{2}$, then at the same point we notice that $ d(\mathcal{F}2, \mathcal{F}4) = d(2, 4) $. Thus it is not satisfy the BCP \cite{sb} and also the contractive condition (iv) of Theorem 3.1 \cite{5}.
\end{rem}
\begin{rem}
	In Example \ref{ex1}, observe that the binary relation $ \mathcal{R} $ is nonreflexive, nonirreflexive, nonsymmetric and nonantisymmetric. Therefore it is none of near-order, partial order, pre-order and tolerance. Thus, it is worth mentioning that Theorem \ref{b} is genuine extension and improvement of Alam and Imdad \cite{5} in b metric space. 
\end{rem}

\end{document}